\documentclass{amsart}
\newtheorem{theorem}{Theorem}[section]
\newtheorem{lemma}[theorem]{Lemma}
\theoremstyle{definition}
\newtheorem{definition}[theorem]{Definition}

\newtheorem{corollary}[theorem]{Corollary}

\theoremstyle{remark}
\newtheorem{remark}[theorem]{Remark}

\numberwithin{equation}{section}

\begin{document}
\def\C{\mathbb C}
\def\R{\mathbb R}
\def\X{\mathbb X}
\def\cA{\mathcal A}
\def\cT{\mathcal T}
\def\Z{\mathbb Z}
\def\Y{\mathbb Y}
\def\Z{\mathbb Z}
\def\N{\mathbb N}
\def\cal{\mathcal}

\title[Asymptotic Behavior of Individual Orbits of Discrete Systems]{Asymptotic Behavior of Individual Orbits of Discrete Systems}
\author{Nguyen Van Minh}
\address{Department of Mathematics, University of West Georgia,
Carrollton, GA 30118. USA} \email{vnguyen@westga.edu}
\thanks{The author is grateful to the anonymous referee for his carefully reading the manuscript and pointing out several inaccuracies and suggestions to improve the presentation of this paper.}


\begin{abstract}
We consider the asymptotic behavior of bounded solutions of the
difference equations of the form $x(n+1)=Bx(n) + y(n)$ in a Banach
space $\X$, where $n=1,2,...$, $B$ is a linear continuous operator
in $\X$, and $(y(n))$ is a sequence in $\X$ converging to $0$ as
$n\to\infty$. An obtained result with an elementary proof says that
if $\sigma (B) \cap \{ |z|=1\} \subset \{ 1\}$, then every bounded
solution $x(n)$ has the property that $\lim_{n\to\infty}
(x(n+1)-x(n)) =0$. This result extends a theorem due to
Katznelson-Tzafriri. Moreover, the techniques of the proof are
furthered to study the individual stability of solutions of the
discrete system. A discussion on further extensions is also given.
\end{abstract}
\date{\today}

\maketitle

\section{Introduction, Notations and Preliminaries} \label{section 1}
Suppose that $T$ is a power-bounded linear continuous operator in a given complex Banach space $\X$, i.e., $\sup_{n\in \N} \| T^n\| <\infty$.
In \cite[Theorem 1]{kattza} it is proven  that
 $\lim_{n\to\infty} \| T^{n+1}-T^n \| =0$ if $\sigma (T)\cap \{ |z|=1\}\subset \{1
\}$. As noted in \cite{vu2}, this assertion is actually equivalent to a little weaker one that for each $x_0\in\X$,
$\lim_{n\to\infty} \| T^{n+1}x_0-T^nx_0 \| =0$ if $\sigma (T)\cap \{ |z|=1\}\subset \{1
\}$.
An elegant proof of this assertion, which we refer to as Katznelson-Tzafriri Theorem, was given in \cite{vu}.
There are numerous works on extensions and applications of this result of which to name a few the reader is referred to e.g. \cite{allran},
\cite{arebathieneu}, \cite{arepru}, \cite{baspry} \cite{batneerab}, \cite{chitom},
\cite{eststrzou2}, \cite{kalmonoletom}, \cite{lauvu}, \cite{min},  \cite{mus}, \cite{nee}, \cite{vu}, and their references.

\medskip
It is the first purpose of this note to extend the
Katznelson-Tzafriri Theorem to
difference equations of the form
\begin{equation}\label{eq}
x(n+1)=Bx (n)+y(n),  \quad x(n)\in \X, n\in \N ,
\end{equation}
where $x(n)\in\X$, $B$ is a linear continuous operator acting in $\X$ that is {\it not necessarily assumed to be
power-bounded}, $y(n)\in  \X$ is a sequence satisfying $\lim_{n\to
\infty} y (n)=0.$ Our main result is Theorem \ref{the main} that is
{\it proven by an elementary method} which can be furthered to study the
stability of individual solutions of (\ref{eq}). A Tauberian
theorem (Theorem \ref{the tec}) is stated and then used to prove
Theorem \ref{the main2} on the asymptotical stability of individual
solutions of (\ref{eq}). This result may be seen as the discrete
version of several results in \cite{arebat,batneerab,lauvu,min3}, and it
complements a result on strong stability of solutions in
\cite{vulyu}. For a more complete account of results and methods in this direction the reader is referred to \cite{arebathieneu,chitom,nee}.

\bigskip
 In this note we will use the following notations: $\N =\{ 1,2,\cdots \}$, $\Z$ -  the set of all integers, $\R$ - the set of reals, $\C$ - the complex plane with $\Re z$ denoting the real part of $z\in\C$, $\X$ - a given  complex Banach space.
  A sequence in $\X$ will be denoted by $(x(n))_{n=1}^\infty$, or, simply by $(x(n))$, and the spaces of sequences
\begin{eqnarray*}
l^\infty (\X) &:=& \{ (x(n))  \subset \X | \
\sup_{n\in\N} \| x (n)\| <\infty \} \\
c_0 &:=& \{ (x(n))\subset \X | \ \lim_{n\to\infty}x(n) =0\}
\end{eqnarray*}
are equipped with sup-norm.
The shift operator $S$ acts in $l^\infty (\X)$ as follows:
\begin{eqnarray*}
Sx(n) = x(n+1),\quad n\in\N , x\in l^\infty (\X) .
\end{eqnarray*}
In this paper, for a complex Banach space $\X$, the space of all
bounded linear operators acting in $\X$ is denoted by $L(\X)$; $\rho
(B), \sigma (B), R\sigma (B), Ran (B)$ denote  the resolvent set,
spectrum, residual spectrum, range of $B\in L(\X)$, respectively. It
is well known that the operator $S$ defined as above is a
contraction. Consider the quotient space $\Y := l^\infty (\X) /c_0 $
with the induced norm. The equivalent class containing a sequence
$x\in l^\infty (\X)$ will be denoted by $\bar x$. Since $S$ leaves
$c_0$ invariant it induces a bounded linear operator $\bar S$ acting
in $\Y$. Moreover, one notes that $\bar S$ is a surjective isometry.
As a consequence, $\sigma (\bar S)\subset \Gamma$, where $\Gamma$
denotes the unit circle in the complex plane. We will use the
following estimate for the resolvent of the isometry $\bar S$ whose
proof can be easily obtained:
\begin{equation}\label{112}
\| R(\lambda ,\bar S)\| \le \frac{1}{||\lambda |-1|}, \quad \mbox{for all} \ |\lambda | \not= 1.
\end{equation}

\section{Main Results}
\subsection{Katznelson-Tzafriri Theorem for Individual Orbits}
Consider the difference equation (\ref{eq}) with $(y(n))\in c_0$. A main result of this note is the following
\begin{theorem}\label{the main}
Let $B$ be any linear continuous operator acting in $\X$ such that $\sigma (B) \cap \Gamma \subset \{ 1\}$, and let $x:=(x(n))_{n=1}^\infty$ be a bounded solution of (\ref{eq}). Then,
\begin{equation}\label{100}
\lim_{n\to\infty} [x(n+1)-x(n)]=0.
\end{equation}
\end{theorem}
 The theorem is an immediate consequence of several lemmas that may be of independent
 interest.


\begin{lemma}\label{lem 2}
 Assume that $\bar x$ is any point in $\Y$, and the complex function $g(\lambda
):= R(\lambda ,\bar S )\bar x$ has the point $\lambda = \xi_0\in \Gamma$ as an isolated singular point. Then, $\xi_0$ is either a removable singular point or a pole of first order.
\end{lemma}
\begin{proof} Without loss of generality we may assume that $\xi_0=1$. Consider $\lambda $ in a
small neighborhood of $1$ in the complex plane. We will express
$\lambda =e^{z}$ with $|z|< \delta_0 $. Choose a small $\delta_0
>0$ such that if $|z|<\delta_0   $, then
\begin{equation}
\frac{1}{|1-|\lambda ||} \le \frac{2}{|\Re z|}.
\end{equation}
Notice that if $0<|z |<\delta _0$, then
\begin{equation}
\| R(\lambda ,\bar S)\bar x \| \le \frac{1}{|1-|\lambda ||} \| \bar
x\| \le \frac{2}{|\Re z|} \| \bar x\| .
\end{equation}
Set $f(z )= R(e^z ,\bar S)\bar x $ with $|z|<\delta_0$. Since $1$ is
a singular point of $\| R(\lambda ,\bar S)\bar x \|$, $0$ is a
singular point of $f(z)$ in $\{ |z|<\delta_0 \}$ . For each $n\in
\Z$ and $0<r<\delta_0$, we have
\begin{eqnarray*}
&&\| \frac{1}{2\pi i} \int_{|z |=r}  z   ^n \left(
1+\frac{ z  ^2}{r^2}\right) f(z)dz \| \\
&& \hspace{2cm} \le \frac{1}{2\pi} \int_{|z |=r} | z  ^n \left(
1+\frac{ z  ^2}{r^2}\right)|\cdot \| f(z)\|\cdot |dz | .
\end{eqnarray*}
If $z =re^{i\varphi}$, where $\varphi$ is real, then one has
\begin{eqnarray}
| z ^n \left( 1+\frac{ z ^2}{r^2}\right)| &=& r^n |1+e^{2i\varphi} |
 =  r^n | e^{-i\varphi}+e^{i\varphi}|\nonumber\\
 &=& r^n 2|cos \varphi|= 2r^{n-1}|\Re z|.
\end{eqnarray}
Therefore,
\begin{eqnarray}
| \frac{1}{2\pi i} \int_{|z |=r}  z ^n \left(
1+\frac{z^2}{r^2}\right) f(z)dz | &\le& \frac{1}{2\pi}
\int_{|z|=r}2r^{n-1} | \Re z | \frac{2}{|\Re  z|} \cdot |dz |
\nonumber \\
&=& \frac{2\cdot 2r^{n-1}}{2\pi} \int_{|z|=r}  |dz | \nonumber \\
&=&{4r^{n}} .\label{15}
\end{eqnarray}
Consider the Laurent series of $f(z )$ at $z =0$,
\begin{equation}\label{laurent exp}
f(z) = \sum_{n=-\infty}^{\infty} a_n z^n,
\end{equation}
where
\begin{equation}\label{laurent coeff}
a_n =\frac{1}{2\pi i} \int_{| z |=r} \frac{f(z)dz }{z ^{n+1}},\quad
n\in \Z .
\end{equation}
From (\ref{15}) it follows that
\begin{eqnarray*}
\| a_{-(n+1)} + r^{-2} a_{-(n+3)}\| &=& \| \frac{1}{2\pi i}\int_{|z
|=r}  z ^n f(z)dz
 + \frac{1}{2\pi i} \int_{|z |=r} \frac{ z ^{n+2}}{r^2}
f(z)dz \| \nonumber\\
&=&  \| \frac{1}{2\pi i} \int_{|z |=r}  z ^n \left(
1+\frac{ z ^2}{r^2}\right) f(z)dz \| \\
&\le& 4r^{n}.
\end{eqnarray*}
Therefore,
\begin{equation}\label{an}
\| r^2a_{-(n+1)} + a_{-(n+3)}\| \le 4r^{n+2}, \quad n\in \Z .
\end{equation}
 Letting
$r$ tend to $0$ in (\ref{an}), we come up with $a_{-k}=0$ for all $k
\ge 2$. This shows that $z=0$ is a removable singular point (when $a_{-1}=0$) or a pole of first order of $f(z )$. This yields that the complex function  $g(\lambda
):= R(\lambda ,\bar S )\bar x$ has $\lambda =1$ as a removable singular point or a pole of first
order. The lemma is proven.
\end{proof}

Before proceeding we introduce a new notation:  let $0\not= z\in \C$ such that
$z= re^{i\varphi}$ with reals $r=|z|,\varphi$, and let $F(z)$ be any complex
function. Then, (with $s$ larger than $r$) we define
\begin{equation}
\lim_{\lambda \downarrow z} F(\lambda ):= \lim_{s\downarrow r}F(s e^{i\varphi}).
\end{equation}


That is, we consider the limit as $\lambda$ approaches $z$ in a special direction corresponding to the ray $\arg \lambda =\arg z$.
\begin{lemma}\label{lem 3}
Let $\xi_0\in\Gamma$ be an isolated singular point of $g(\lambda )= R(\lambda ,\bar S)\bar x$ with a given $\bar x\in \Y$. Then, this singular point $\xi_0$ is removable provided that
\begin{equation}\label{erg}
\lim_{\lambda \downarrow \xi_0} (\lambda -\xi_0 )R(\lambda ,\bar S)\bar x =0 .
\end{equation}
\end{lemma}
\begin{proof}
As shown in Lemma \ref{lem 2}, $\xi$ is either a removable singular point or pole of first order. Without loss of generality we may assume that $\xi_0=1$ for the reader's convenience. Then, the Laurent series of $g(\lambda)$ is of the form
\begin{equation}
g(\lambda ) = \sum_{n=0}^\infty (\lambda -1)^n b_n + \frac{1}{\lambda -1} b_{-1}.
\end{equation}
We need to show that under condition (\ref{erg}) the coefficient $b_{-1}=0$. In fact,
\begin{eqnarray*}
0=\lim_{\lambda \downarrow 1}g(\lambda ) &=& \lim_{\lambda
\downarrow 1}(\lambda -1 )R(\lambda
,\bar S)\bar x  \\
&=&\lim_{\lambda \downarrow 1} (\lambda -1 )\left( \sum_{n=0}^\infty
(\lambda
-1)^n b_n + \frac{1}{\lambda -1} b_{-1} \right)\\
&=&\lim_{\lambda \downarrow 1}  \left( \sum_{n=0}^\infty (\lambda
-1)^{n+1} b_n + \frac{\lambda -1}{\lambda -1} b_{-1} \right)\\
&=& b_{-1}.
\end{eqnarray*}
This shows $\xi_0$ is removable. The lemma is proven.
\end{proof}


\begin{definition}
Let $(x(n))$ be a bounded sequence in $\X$.
The notation $\sigma (x)$ stands
for the set of all non-removable singular points of the complex
function $g(\lambda ):= R(\lambda ,\bar S)\bar x$. This set may be
referred to as {\it the spectrum of $x$}, an analog of a similar concept in \cite{arebat}. Obviously, $\sigma (x
)$ is a closed subset of $\Gamma$.
\end{definition}

\begin{lemma}\label{lem 1} Let $x:=(x(n))$ be a bounded solution of
equation (\ref{eq}). Then,
\begin{equation}
\sigma (x) \subset \sigma (B) \cap \Gamma .
\end{equation}
\end{lemma}
\begin{proof}
Consider $R(\lambda ,\bar S )\bar x$ for all $|\lambda |\not= 1$.
Since $x$ is a bounded solution of (\ref{eq}) and $\bar y=0$ we have
\begin{eqnarray}
 R(\lambda ,\bar S )\bar S \bar x&=&  R(\lambda ,\bar S )\bar B \bar x + R(\lambda ,\bar S ) \bar y \nonumber \\
&=& \bar B R(\lambda ,\bar S )\bar x \label{eq2}
\end{eqnarray}
On the other hand, the identity $ \lambda R(\lambda ,\bar S ) \bar x
-\bar x  =  R(\lambda ,\bar S )\bar S \bar x$ gives
\begin{equation}
\lambda R(\lambda ,\bar S ) \bar x -\bar x =  R(\lambda ,\bar S
)\bar S \bar x    =\bar B R(\lambda ,\bar S )\bar x,
\end{equation}
so,
\begin{eqnarray*}
\bar x &=& \lambda R(\lambda ,\bar S ) \bar x - \bar B R(\lambda ,\bar S )\bar x \\
&=& (\lambda   - \bar B ) R(\lambda ,\bar S )\bar x .\\
\end{eqnarray*}
Obviously, $R(\lambda ,\bar S )\bar x$ is analytic on $\C \backslash
\Gamma$. Moreover, if $|\xi_1 |=1$ and $\xi_1\not\in \sigma (B)\cap
\Gamma$, (as we can easily check that $\sigma (\bar B )=\sigma
(B)$), in a small neighborhood $U(\xi_1 )$ of $\xi_1$ we have
\begin{eqnarray}\label{eq3}
R(\lambda ,\bar S )\bar x &=& (\lambda   - \bar B )^{-1} \bar x
,\quad \lambda \in U(\xi_1)\backslash \Gamma .
\end{eqnarray}
This shows that $g(\lambda ) = R(\lambda ,\bar S )\bar x $ is
analytically extendable to a neighborhood of $\xi_1$, that is,
$\xi_1 \not\in \sigma (x)$. The lemma is proven.
\end{proof}

\bigskip\noindent {\it Proof of Theorem \ref{the main}}:
 The identity $R(\lambda ,\bar S)\bar
S\bar x = \lambda R(\lambda , \bar S)\bar x - \bar x$ gives
\begin{eqnarray*}
 R(\lambda ,\bar S)(\bar S \bar x -\bar x ) &=& R(\lambda ,\bar S)\bar S \bar x-R(\lambda ,\bar S) \bar x \\
 &=& \lambda R(\lambda , \bar S)\bar x - \bar x -R(\lambda ,\bar S) \bar x    \\
 &=& (\lambda -1 )R(\lambda ,\bar S)  \bar x -\bar x .
\end{eqnarray*}
Therefore,
\begin{eqnarray}\label{eq8}
h(\lambda ):= (\lambda -1) R(\lambda ,\bar S)(\bar S \bar x -\bar x
) &=& (\lambda -1 )^2 R(\lambda ,\bar S)  \bar x -(\lambda
-1)\bar x .
\end{eqnarray}
By Lemmas \ref{lem 1}, \ref{lem 2}, $\sigma (Sx-x) \subset \sigma
(B)\cap \Gamma \subset \{ 1\}$, $h(\lambda )$ is extendable
analytically to the whole complex plane with only possible exception
at $1$. Since $g(\lambda ): = R(\lambda ,\bar S) \bar x$ has $1$ as
a either removable singular point or a pole of first order we have
\begin{eqnarray*}
 \lim_{\lambda \to 1} (\lambda -1)^2 R(\lambda ,\bar S) \bar x   &=&  0.
\end{eqnarray*}
Consequently,
\begin{eqnarray*}
 \lim_{\lambda \to 1} (\lambda -1) R(\lambda ,\bar S)(\bar S \bar x -\bar x ) &=&   \lim_{\lambda \to 1} [ (\lambda -1 )^2 R(\lambda ,\bar S) \bar x -(\lambda -1)\bar x
 ]\\
 &=&    \lim_{\lambda \to 1}   (\lambda -1 )^2 R(\lambda ,\bar S) \bar x -\lim_{\lambda \to 1} (\lambda -1)\bar x \\
 &=& 0.
\end{eqnarray*}
By Lemma \ref{lem 3}, $h(\lambda)$ has $\lambda =1$ as a removable
singular point, so $h(\lambda )$ is extendable to an entire
function. For $|\lambda |>1$, by (\ref{112})
we have
\begin{eqnarray*}
\limsup_{|\lambda|\to\infty} \| h(\lambda )\| &=&  \limsup_{|\lambda|\to\infty} \| (\lambda -1) R(\lambda ,\bar S)(\bar S \bar x -\bar x )\| \nonumber\\
 &\le& \limsup_{|\lambda|\to\infty}  \frac{|\lambda |+1}{  |\lambda |-1 } \cdot \| \bar S \bar x -\bar x \|  \nonumber\\
 &=& \| \bar S \bar x -\bar x \| .\label{eq9}
\end{eqnarray*}
This shows that $h(\lambda )$ is bounded on the complex plane, so,
as a bounded entire function it should be a constant by Liouville's
Theorem. In turn, it is identically equal to zero because $h(1
):=\lim_{\lambda \to 1}h(\lambda )=0$. Since $R(\lambda ,\bar S)$ is
injective for each $\lambda \not= 1$, we have $\bar S\bar x-\bar
x=0$. Therefore, $(Sx-x)\in c_0$, that is, (\ref{100}). The theorem
is proven.
\begin{remark}
In the remark following Theorem \ref{the tec} we will give an alternative proof of Theorem \ref{the main} in a more general context. However, the above proof seems to be more elementary.
\end{remark}

\subsection{Stability of Individual Orbits}

 We define ${\cal M}_{\bar x}$ as the smallest closed subspace
of $\Y := l^\infty (\X)/c_0$ spanned by $\{ \bar S^n\bar x, n\in
\Z\}$. Consider the restriction $\bar S |_{{\cal M}_{\bar x}}$ that
is also a surjective isometry.

\begin{lemma}\label{lem 4} Let $x:=(x(n))\in l^\infty (\X )$. Then, the following assertions hold:
\begin{enumerate}
\item $\sigma (x)=\emptyset$ if and only if $x\in c_0$;
\item If $\sigma (\bar x) \not= \emptyset$, then
  $\sigma (x)=\sigma ( \bar S |_{{\cal M}_{\bar x}})$.
\end{enumerate}
\end{lemma}
\begin{proof}
(i): If $\sigma (x)=\emptyset$, the function $g(\lambda ):=R(\lambda ,\bar S)\bar x$ can be extended to an entire function. Using exactly the argument in the proof of Theorem \ref{the main} we come up with the boundedness of the complex function $t(\lambda ):= (\lambda -1) R(\lambda ,\bar S)\bar x$, so by Liouville's Theorem $t(\lambda )$ is a constant. And thus, $t(\lambda )=\lim_{\lambda\to 1} (\lambda -1)g(\lambda )= 0$. The injectiveness of $R(\lambda ,\bar S)$ for each $|\lambda|\not =1$ yields that $\bar x=0$. The converse is clear.

\medskip\noindent
(ii): By (i), $\bar x\not= 0$, so $\rho (\bar S |_{{\cal M}_{\bar x}})\not= \emptyset$.
Let $\xi_0\in \rho (\bar S |_{{\cal M}_{\bar x}})$. Then, since for
$|\lambda |\not= 1$
$$
R(\lambda ,\bar S)\bar x = R(\lambda ,\bar S |_{{\cal M}_{\bar
x}})\bar x
$$
it is clear that $\xi_0$ is a regular point of $g(\lambda )$.

\medskip
Conversely, let $\xi_0$ be a regular point of $g(\lambda)$. Without
loss of generality we may assume $|\xi_0 |=1$, otherwise it is
already in $\rho (\bar S |_{{\cal M}_{\bar x}})$. We will show that
$\xi_0\in \rho (\bar S |_{{\cal M}_{\bar x}})$ by proving that the
equation
\begin{equation}\label{20}
\xi_0 v-\bar S v = w
\end{equation}
has a unique solution $v\in {\cal M}_{\bar x}$ for each given $w\in
{\cal M}_{\bar x}$. First, we show that there is at least one solution.
In fact, we note that for each $n\in\Z$ the set of regular points of
$g(\lambda )=R(\lambda ,\bar S)\bar x$ is the same as that of $\bar
S^n g(\lambda )=R(\lambda ,\bar S)\bar S^n\bar x$. And in turn, by
the property of holomorphic functions, the set of all regular points
of $g(\lambda )=R(\lambda ,\bar S)\bar x$ must be part of that of
the function $k(\lambda )=R(\lambda ,\bar S)w$, so $k(\lambda
)=R(\lambda ,\bar S)w$ is analytically extendable to a neighborhood
of $\xi_0$. In particular, $\lim_{\lambda \to\xi_0} k(\lambda )=v\in
{\cal M}_{\bar x}$, so
\begin{eqnarray*}
\lim_{(|\lambda|>1), \ \lambda \to \xi_0 }[\lambda R(\lambda ,\bar S)w - R(\lambda ,\bar S)\bar S w] &=& w\\
\xi_0v -\bar S v&=&w.
\end{eqnarray*}
To show that equation (\ref{20}) has a unique solution in ${\cal
M}_{\bar x}$ we can show that the homogeneous equation $\xi_0 v-\bar
Sv=0$ has only a trivial solution in ${\cal M}_{\bar x}$. In fact, let
$v_0\in {\cal M}_{\bar x}$ be a solution of this equation. Then, for
each $|\lambda|>1$, using the identity $R(1
,A)=(I-A)^{-1}=\sum_{n=0}^\infty A^n$ for each $\|A\| <1$ and $\bar
S^n v_0=\xi_0^n v_0$ we have
\begin{eqnarray}
R(\lambda,\bar S) v_0 &=& \sum_{n=0}^\infty \frac{1}{\lambda^{n+1}} \bar S ^n v_0\nonumber\\
&=& \sum_{n=0}^\infty \frac{1}{\lambda^{n+1}} \xi_0 ^n v_0\nonumber\\
&=& \frac{1}{\lambda-\xi_0}v_0.\label{30}
\end{eqnarray}
Since $v_0\in {\cal M}_{\bar x}$, this function must, as above, be extendable analytically to a neighborhood of $\xi_0$, and this is possible only if $v_0=0$. Summing up, we have that $\xi_0\in\rho ({\cal M}_{\bar
x})$, so the lemma is proven.
\end{proof}


\begin{theorem}\label{the tec}
Let $(x(n))$ be a bounded sequence such that the set $\sigma ( x)$ of
all non-removable singular points of $g(\lambda )=R(\lambda ,\bar
S)\bar x$ is countable, and let the following condition holds for
each $\xi_0\in \sigma (\bar x)$
\begin{equation}
\lim_{\lambda \downarrow \xi_0} (\lambda -\xi_0 )R(\lambda ,\bar
S)\bar x =0 .
\end{equation}
Then, \begin{equation}\label{55} \lim_{n\to\infty}x(n)=0.
\end{equation}
\end{theorem}
\begin{proof}
We have to show (\ref{55}), that is, $\bar x= 0$, or equivalently,
${\cal M}_{\bar x}$ is trivial. Suppose to the contrary that it is
not. Then, by Lemma \ref{lem 4}, $\sigma (x)= \sigma (\bar S |_{{\cal
M}_{\bar x}})\not= \emptyset$. Since $\sigma (x)$ is a non-empty
closed subset of $\Gamma$ and is countable, it has
an isolated point, say $\xi_0$, so $\xi_0$ is an isolated
singular point for $g(\lambda )$. By Lemma \ref{lem 2} this isolated
singular point must be either a removable singular point or a pole
of first order. Since $\xi_0$ is a pole of first order of the resolvent $R(\lambda , \bar S |_{{\cal
M}_{\bar x}})$, by a well known result in Functional Analysis\footnote{We
actually avoid applying Geldfand Theorem in this case.}  (see e.g.
\cite[Theorem 5.8 A, p. 306]{tay}, or, \cite[Theorem 3, p.
229]{yos}) $\xi_0$ must be an eigenvalue of $\bar S |_{{\cal
M}_{\bar x}}$ with a non-zero eigenvector $w_0$. As in the proof of
Lemma \ref{lem 4}, (see \ref{30})), for each $|\lambda|\not= 1$ we
have
\begin{equation}\label{min2}
R(\lambda ,\bar S)w_0 = \frac{1}{\lambda -\xi_0}w_0.
\end{equation}
On the other hand, by Lemma \ref{lem 3} $\xi_0$ is a removable
singular point for $g(\lambda )$, so is for $R(\lambda ,\bar S)w_0$.
This is possible only if $w_0=0$, contradicting that
$w_0$ is a non-zero vector. This proves the theorem.
\end{proof}

\begin{remark}
An alternative proof of Theorem \ref{the main} is a direct application of Lemma \ref{lem 1} and Theorem \ref{the tec}. As another consequence of Theorem \ref{the tec} we have the
following on the strong asymptotical stability of solutions of
(\ref{eq}).
\end{remark}

\begin{theorem}\label{the main2}
For equation (\ref{eq}) assume that $(y(n))\in c_0$, and the operator $B$ in equation (\ref{eq}) has $\sigma (B)\cap \Gamma$ as a countable set. Then, the following holds for each bounded solution $(x(n))$ of (\ref{eq})
\begin{equation}
\lim_{n\to\infty} x(n)=0,
\end{equation}
provided that for each $\xi_0 \in \sigma (B)\cap \Gamma$ the following condition holds
\begin{equation}\label{min}
\lim_{\lambda  \downarrow \xi_0} (\lambda -\xi_0)R(\lambda , \bar
S)\bar x =0 .
\end{equation}
\end{theorem}
\begin{proof}
This theorem is an immediate consequence of Lemma \ref{lem 1} and Theorem \ref{the tec}.
\end{proof}

\section{Discussion}

Theorem \ref{the main} may be seen as an extension of the following
result due to Katznelson-Tzafriri (see \cite[Theorem 1]{kattza}).
\begin{theorem}
Let $T$ be a power bounded linear operator in a Banach space $\X$ such that $\sigma (T)\cap \Gamma \subset \{ 1\}$. Then,
\begin{equation}\label{kt2}
\lim_{n\to\infty} (T^{n+1}-T^n )=0.
\end{equation}
\end{theorem}
In fact, as noted in \cite{vu2} this theorem is equivalent to a
weaker one
\begin{theorem}\label{the vu}
Let $T$ be a power bounded linear operator in a Banach space $\X$
such that $\sigma (T)\cap \Gamma \subset \{ 1\}$. Then, for each
$x_0\in\X$
\begin{equation}\label{kt}
\lim_{n\to\infty} (T^{n+1}x_0-T^n x_0)=0.
\end{equation}
\end{theorem}
Obviously, our Theorem \ref{the main} extends Theorem \ref{the vu}.

\medskip
As an immediate consequence of Theorem \ref{the main2} we have the
following corollary:
\begin{corollary}\label{lem 3.3}
Let $B\in L(\X)$ be a power bounded operator such that $\sigma (B)\cap \Gamma$ is a countable set.
Moreover, assume that for each $\xi_0\in \sigma (B)\cap \Gamma$ the following holds for each $x_0\in \X$
\begin{equation}\label{con b}
\lim_{\lambda  \downarrow \xi_0} (\lambda -\xi_0 )R(\lambda ,  B)x_0 =0 .
\end{equation}
Then, for every $x_0\in\X$
\begin{equation}
\lim_{n\to\infty} B^nx_0=0 .
\end{equation}
\end{corollary}
\begin{proof}
Let $x(n)=B^nx_0$. Then, $(x(n))$ is a bounded solution of
(\ref{eq}) with $(y(n))=0$. Therefore, if $|\lambda |>1$, $\lambda
\in \rho (B)$ and $\lambda \in \rho (\bar S)$, so by (\ref{eq3}) (and the proof of
Lemma \ref{lem 1}),
\begin{eqnarray}
\lim_{\lambda  \downarrow \xi_0}\|   (\lambda -\xi_0 )R(\lambda ,
\bar S)\bar x \| &=& \lim_{\lambda  \downarrow \xi_0}\|   (\lambda
-\xi_0 )R(\lambda ,
\bar B)\bar x \|  \\
&\le &
\lim_{\lambda  \downarrow \xi_0} \sup_{n\in\N}\{ \|   (\lambda -\xi_0 )R(\lambda ,    B) B^n x_0 \| \}\nonumber \\
&\le& \lim_{\lambda  \downarrow \xi_0} \sup_{n\in\N}\{ \| B^n\|\} \cdot \|   (\lambda -\xi_0 )R(\lambda ,    B) x_0 \| \nonumber \\
&\le& \lim_{\lambda  \downarrow \xi_0}     \| (\lambda -\xi_0
)R(\lambda ,    B) x_0 \| =0.
\end{eqnarray}
Therefore, by Theorem \ref{the main2}, $x(n)=B^nx_0 \to 0$.
\end{proof}

\begin{remark}
{\it Condition (\ref{con b}) is satisfied if} $R\sigma (B) \cap \Gamma
=\emptyset$, and hence Corollary \ref{lem 3.3} yields the discrete version of the Arendt-Batty-Ljubich-Vu Theorem \cite[Theorem 5.1]{arebat3},
\cite[Corollary 3.3]{eststrzou2}, \cite{vulyu}).
In fact, since $B$ is power-bounded one can easily show that there exists a positive constant $C$ such that
\begin{equation}\label{n re}
\| R(\lambda ,B) \| \le \frac{C}{||\lambda |-1|}, \quad  \mbox{for} \ |\lambda|>1 .
\end{equation}
Next, since $R\sigma (B) \cap \Gamma
=\emptyset$, for
all $\xi_0 \in \sigma (B)\cap \Gamma$, the range of $(\xi_0 -B) $
is dense in $\X$. Therefore, for each $ x_0 \subset \X$ there is a sequence $(x_0^n)\in Ran (\xi_0-B)$ such that $ x_0=\lim_{n\to\infty}x_0^n$. Then, $x^n_0=(\xi_0-B)y^n_0$ for some
sequence $(y^n_0)\subset \X$.
By our definition of the limit as $\lambda \downarrow \xi_0$
we have $|\lambda -\xi_0| = ||\lambda |-|\xi_0|| = ||\lambda |-1| \to 0$, so in view of (\ref{n re}), for each fixed $n$ we have
\begin{eqnarray}
\lim_{\lambda  \downarrow \xi_0} \|(\lambda -\xi_0 )R(\lambda ,
B)x^n_0 \|&=&
\lim_{\lambda  \downarrow \xi_0} \| (\lambda -\xi_0 )R(\lambda ,  B)(\xi_0 -B)y^n_0 \| \nonumber   \\
&=& \lim_{\lambda  \downarrow \xi_0}\|  (\lambda -\xi_0 )R(\lambda ,
B)[(\lambda -B)y^n_0 +(\xi_0 -\lambda )y^n_0]\| \nonumber\\
&=&  \lim_{\lambda  \downarrow \xi_0}\left(|\lambda -\xi_0| \cdot \| y^n_0\| +   \| (\lambda -\xi_0
)R(\lambda
, B)(\xi_0 -\lambda )y^n_0\| \right) \nonumber\\
&\le& 0+\lim_{\lambda  \downarrow \xi_0} \| \lambda -\xi_0 \|^2 \cdot
\| R(\lambda , B) \|\cdot \| y^n_0\| \nonumber\\
&\le& \lim_{\lambda  \downarrow \xi_0} ||\lambda|-|\xi_0||^2\cdot\frac{C  }{||\lambda |-1|}\cdot \| y^n_0\|\nonumber\\
&=&  \lim_{\lambda  \downarrow \xi_0} \frac{ ||\lambda|-1|^2\cdot C}{||\lambda |-1|}\cdot \| y^n_0\| =0 . \label{rem1}
\end{eqnarray}
By (\ref{n re}) for every fixed $n$
\begin{equation}
\limsup_{\lambda  \downarrow \xi_0} \|(\lambda -\xi_0 )R(\lambda ,
B)  (x_0-x^n_0 )\|  \le C\| x_0-x_0^n\| .\label{rem2}
\end{equation}
Finally, for each $n\in \N$ from (\ref{rem1}) and (\ref{rem2}) we
have
\begin{eqnarray*}
\limsup_{\lambda  \downarrow \xi_0} \|(\lambda -\xi_0 )R(\lambda ,
B)x_0 \| &\le&   \limsup_{\lambda  \downarrow \xi_0} [
\|(\lambda -\xi_0 )R(\lambda , B)x^n_0 \\
&& +(\lambda -\xi_0 )R(\lambda , B)(x_0-x^n_0 ) \| ] \\
& \le&  C\| x_0-x_0^n\| .
\end{eqnarray*}
Since $\| x_0^n-x_0\| \to 0$ as $n\to\infty$, we have that (\ref{con
b}) holds for any $x_0\in\X$.
\end{remark}

\medskip
Let us define a so-called  {\it Condition H}  for a closed subspace ${\cal M}$ of $l^\infty (\X)$ by the following axioms:
\begin{enumerate}
\item ${\cal M}$ is bi-invariant under translation $S$, that is, ${\cal M}=\{ x\in l^\infty (\X): Sx \in {\cal M}\}$;
\item If $x:=(x(n))\in {\cal M}$ and $A\in L(\X)$, then $y:= (Ax(n))\in  {\cal M}$;
\item $c_0\subset  {\cal M}$,
\end{enumerate}
As an example of such a closed subspace ${\cal M}$ of $l^\infty
(\X)$ that satisfies Condition H one can take the space $AAP(\N
,\X)$ of all asymptotic almost periodic sequences. If we replace
$c_0$ by ${\cal M}$, we will arrive at various analogs of
Theorems \ref{the main}, \ref{the main2} and \ref{the tec}. Note
that the proofs of these analogs are identically similar to those of
the mentioned theorems. Below are the statements of analogs of the
mentioned theorems in case ${\cal M}=AAP(\N,\X)$.

\bigskip
Recall that a sequence $(x(n))$ is said to be asymptotically almost
periodic if $x(n)=y(n)+z(n)$ for all $n\in \N$ where $(y(n))\in c_0$
and $(z(n))$ is an almost periodic sequence. An almost periodic
sequence on $\N$ is the restriction to $\N$ of an almost periodic
sequence on $\Z$. In turn, an almost periodic sequence on $\Z$ is
defined to be an element of the following subspace $\overline{span\{
(\lambda ^{n} y_0)_{n\in \mathbb Z},\lambda \in \Gamma ,y_0 \in
\mathbb X\}}$  of $l^\infty (\X)$. In the following, by abusing
notations, $\bar x$ denotes the equivalent class of $l^\infty
(\X)/AAP(\N,\X)$ containing $x$, $\bar S$ denotes the the operator
acting in $l^\infty (\X)/AAP(\N,\X)$ induced by $S$.

\begin{theorem}\label{the main4}
Let $B$ be any linear continuous operator acting in $\X$ such that
$\sigma (B) \cap \Gamma \subset \{ 1\}$, and let
$x:=(x(n))_{n=1}^\infty$ be a bounded solution of (\ref{eq}) in
which $(y(n))\in AAP(\N,\X)$. Then, the sequence $(y(n))$, defined
as $y(n):= x(n+1)-x(n)$ for all $n\in \N$, is asymptotically almost
periodic.
\end{theorem}

\begin{theorem}\label{the tec2}
Let $(x(n))$ be a bounded sequence such that the set
$\sigma_{AAP(\N,\X)} ( x)$ of all non-removable singular points of
$g(\lambda )=R(\lambda ,\bar S)\bar x$ is countable, and let the
following condition hold for each $\xi_0\in \sigma_{AAP(\N,\X)} ( x)$
\begin{equation}
\lim_{\lambda \downarrow \xi_0} (\lambda -\xi_0 )R(\lambda ,\bar
S)\bar x =0 .
\end{equation}
Then, $(x(n))$ is asymptotically almost periodic.
\end{theorem}

\begin{theorem}\label{the main7}
For equation (\ref{eq}) assume that $(y(n))\in AAP(\N,\X)$, and the
operator $B$ in equation (\ref{eq}) has $\sigma (B)\cap \Gamma$ as a
countable set. Then, each bounded solution $(x(n))$ of (\ref{eq}) is
asymptotically almost periodic, provided that for each $\xi_0 \in
\sigma (B)\cap \Gamma$ the following holds
\begin{equation}\label{min6}
\lim_{\lambda \downarrow \xi_0} (\lambda -\xi_0)R(\lambda , \bar
S)\bar x =0 .
\end{equation}
\end{theorem}

\bibliographystyle{amsplain}

\end{document}